\documentclass[12pt]{amsart}
\usepackage[english]{babel}
\usepackage[latin1]{inputenc}

\usepackage{graphicx}
\usepackage{amssymb, amsmath}
\usepackage{enumerate,centernot}
\usepackage[colorlinks=true,linkcolor=blue,urlcolor=blue,citecolor=blue]{hyperref}

\newtheorem{theorem}{Theorem}[section]
\newtheorem{definition}[theorem]{Definition}

\newtheorem{corollary}[theorem]{Corollary}

\theoremstyle{definition}

\newtheorem{remark}[theorem]{Remark}

\newcommand{\rn}{\mathbb R^n}
\newcommand{\rd}{\mathbb R^d}

\newcommand{\ed}{\color{black}} 

\newcommand{\KK}{\mathcal{K}}

\def\avgint_#1{\mathchoice{\mathop{\kern 0.2em\vrule width 0.6em height 0.69678ex depth -0.58065ex \kern -0.8em \intop}\nolimits_{\kern -0.4em#1}}{\mathop{\kern 0.1em\vrule width 0.5em height 0.69678ex depth -0.60387ex \kern -0.6em \intop}\nolimits_{#1}} {\mathop{\kern 0.1em\vrule width 0.5em height 0.69678ex depth -0.60387ex \kern -0.6em \intop}\nolimits_{#1}} {\mathop{\kern 0.1em\vrule width 0.5em height 0.69678ex depth -0.60387ex \kern -0.6em \intop}\nolimits_{#1}}}

\begin{document}

\title[Extrapolation of compactness on mixed Lebesgue spaces]{Extrapolation of compactness on weighted mixed Lebesgue spaces}

\author[M. J. Carro]{Mar\'{\i}a J. Carro}
\address{Mar\'{\i}a J. Carro, Department of Analysis and Applied Mathematics, Complutense University  of Madrid, 28040 Madrid, Spain.} 
\email{mjcarro@ucm.es}

\author[C. P\'erez]{Carlos P\'erez}
\address{Carlos P\'erez, Department of Mathematics
University of the Basque Country/Euskal Herriko Unibertsitatea, 48080 Bilbao
Spain. } 
\email{cperez@bcamath.org}

\author[R.H. Torres]{Rodolfo H. Torres}
\address{Rodolfo H. Torres, Department of Mathematics, University of California, Riverside, CA 92521, USA.} 
\email{rodolfo.h.torres@ucr.edu}

\thanks{The first  author has been  partially supported by grants PID2020-113048GB-I00 funded by MCIN/AEI/10.13039/501100011033,  CEX2019-000904-S funded by MCIN/AEI/ 10.13039/501100011033 and Grupo UCM-970966 (Spain).
The second author is supported by grant  PID2020-113156GB-I00, Spanish Government; by the Basque Government through grant 
IT1615-22 and the BERC 2014-2017 program and by BCAM Severo Ochoa accreditation SEV-2013-0323, Spanish Government. }



\begin{abstract}    A version of the Fr\'echet-Kolmogorov theorem for the compactness of operators in weighted mixed Lebesgue spaces   is used to develop a  corresponding compact extrapolation theory a la Rubio de Francia. Several applications are presented too. 
\end{abstract}

\maketitle

\bigskip
\section{Introduction}

The extrapolation result by Rubio de Francia has become a powerful tool for extending the weighted boundedness of an operator from $L^p(w)$ for every $w\in A_p$ to boundedness from $L^q(w)$ for every $w\in A_q$ (\cite{rf:rf, d:d, cmp2}). 

More recently, several authors have investigated extending Rubio de Francia's extrapolation result to the category of compact operators on weighted Lebesgue spaces. Contributions from Cao, Olivo, and Yabuta \cite{coy:coy})  along with Hyt\"onen and Lappas \cite{hl:hl} (see also  \cite{xyy:xyy}), have established linear, multilinear, and off-diagonal results. The two groups of authors adopt different approaches: while both use a combination of interpolation and extrapolation techniques, the former relies on the classical Fr\'echet-Kolmogorov theorem to characterize precompactness, whereas the latter employs more abstract arguments concerning compact operators. Generally speaking, these works   show that if an operator is compact on a specific initial weighted Lebesgue space, then it is also compact on all corresponding weighted Lebesgue spaces with Muckenhoupt weights.   A more general version of the extrapolation result in the setting of Banach Function Spaces (BFS) satisfying certain properties with respect to the Hardy-Littlewood maximal function was develop by Lorist and Nieraeth \cite{LN:LN}.

In this article, we build upon the results in \cite{cst:cst}, where a notion of uniform compactness was introduced and applied via extrapolation to derive new compactness results for pseudodifferential operators. 
While this assumption of uniform compactness is stronger than the assumptions used in other places in the literature, it provides a very straightforward way to extrapolate compactness via an application of a modern form of the Rubio de Francia extrapolation of continuity result,  is easy to verify in applications, and yields also a stronger uniform compactness conclusion. We will extend these ideas to the context of weighted mixed Lebesgue spaces, defined as follows

\begin{definition}
Given  $1\leq p,q <\infty$ and two weights  $u$ in $ \mathbb R^n$ and  $v$ in  $ \mathbb R^m$, we define the space 
$$
L^p_uL^q_v(\mathbb R^{n+m})=\{ f:\mathbb R^{n+m} \to \mathbb C  \mbox{ measurable}: \|f\|_{L^p_uL^q_v(\mathbb R^{n+m})}<\infty\},$$
where 
$$
\|f\|_{L^p_uL^q_v(\mathbb R^{n+m})}=\left(\int_{\mathbb R^{n}} \left[ \int_{\mathbb R^m }f(x,y)^q v(y)dy\right]^{p/q} u(x) dx \right)^{1/p}.
$$
\end{definition}
The unweighted version of these spaces was  introduced and thoroughly  studied by  A. Benedek and R. Panzone  in \cite{bp:bp}.  See \cite{k:k} and the reference therein for further properties in the weighted case.

 When the dimensions are clear from the context or do not need to be specified, we will simply write $L^p_uL^q_v$. We observe that if $p=q$, then $L^p_uL^q_v=L^p(w)$ is the standard Lebesgue space on $\mathbb R^{n+m}$ with weight $w(x, y)=u(x)v(y)$.

\section{Preliminaries }

\subsection{Weights}
By a weight we mean a non-negative, measurable, and locally integrable function. We recall
that a weight $w$ belongs to the class $A_p$,   with $1< p < \infty$ and $\frac1p+\frac1{p'}=1$ , if
$$ [w]_{A_p}= \sup_Q \left( \avgint_Q w(y)\, dy \right) \left(
\avgint_Q w(y)^{ 1 - p^{\prime} }\, dy \right)^{p-1}
<\infty.$$

  As it is usually done,   we have used here the notation $\avgint_E f \ dx=f_E=\frac{1}{|E|}\int_E f \ dx$ for the average of $f$ over $E$ with respect to the Lebesgue measure.   Likewise, for   a given measure $\mu$ defined for every cube $Q$, we will write $f_{Q,\mu} = \avgint_Q f d\mu:= \frac{1}{\mu(Q)}\int_Q fd\mu$. In the particular case of densities given by a weight $w$, we will write $f_{Q,w}=\frac{1}{w(Q)}\int_Q fwdx$.

  The number $ [w]_{A_p}$ above   is called the $A_p$ constant of $w$. A weight $w$
belongs to the class $A_1$ if there is a constant $C$ such that
$$ \avgint_Q w(y)\, dy \le C\inf_{Q} w,
$$
and the infimum of these constants $C$ is called the $A_1$ constant
of $w$.

We will need the following well known property of the $A_p$ weights.   Using H\"older's inequality with $p$ and its conjugate $p'$, we have that for every cube $Q$ and every $g\geq 0$,
\begin{equation}\label{eq:prop-Ap-function}
\avgint_Q g\ dx\leq[w]^{\frac{1}{p}}_{A_p} \left( \frac1{w(Q)} \int_Q  g^p\,wdx\right )^{\frac{1}{p}}.
\end{equation}
Specializing inequality \eqref{eq:prop-Ap-function} for $g\equiv\chi_{E}$ we obtain that, for any measurable set $E\subset Q$, that
\begin{equation}\label{eq:prop-Ap-set}
w(Q)
\leq [w]_{A_p}
\left (\frac{|Q|}{|E|} \right )^p w(E)
\end{equation}

\ed

Since the $A_p$ classes are increasing with respect to $p$,
the $A_{\infty}$ class of weights is defined in a natural way by
$$A_{\infty} = \cup_{p>1} A_p, 
$$ 
and  it is characterized by means of this constant 
\begin{equation*}
[w]_{A_\infty}:=\sup_Q\frac{1}{w(Q)}\int_Q M(w\chi_Q )\ dx,
\end{equation*}
as considered in \cite{HP} (see also \cite{HPR}).   In fact we will use the following ``precise openness property" in the proof of Theorem \ref{CommCompa} 
whose proof can be found in \cite{HPR}. 

Let $1<p<\infty$ and let $w\in
A_p$. 
Then  $w\in A_{p-\varepsilon} $ where 
\begin{equation}\label{rprime}
\varepsilon :=\frac{p-1}{(r_\sigma)' }= \frac{p-1}{ 2^{n+1} [\sigma]_{A_{\infty}} }, 
\end{equation}
and, as usual, $\sigma=w^{1-p'}$. Furthermore,
\begin{equation}\label{eq:Ap-e}
 [w]_{A_{p-\varepsilon}} \leq 2^{p-1}[w]_{A_{p}}
\end{equation}

The proof of this result is based in the  following ``precise reverse H\"older property":  Let
\begin{equation}\label{romega} 
r_w:=1+\frac{1}{2^{n+1}[w]_{A_\infty}-1}.
\end{equation}
Then for any cube $Q$, we have that
\begin{equation}\label{eq:RHI}
\left ( \avgint_Q  w^{r_w}\,dx \right )^{\frac1{r_w}} \le 2  \avgint_Q w\,dx.
\end{equation}

In the proof of Theorem \ref{CommCompa} we will be considering weights $w$ in the class $A_2$ and then in this case we can replace the $\varepsilon$ from \eqref{eq:Ap-e} by 
$$
\varepsilon := \frac{1}{ 2^{n+1}[w]_{2} } 
$$
since $\sigma=w^{-1}$ and, in this case,  $[\sigma]_{A_{\infty}}\leq [\sigma]_{A_{2}}=[w]_{A_{2}}$.

A well-known result obtained by Muckenhoupt \cite{m:m} is that the
Hardy-Littlewood maximal function,
$$
Mf(x)= \sup_{Q\ni x} \frac{1}{Q}\int_Q |f(y)| dy,
$$
satisfies  $M:L^p(w)\to L^p(w)$ if and only if  $w$ is in  $A_p$.

\subsection{Sharp maximal operators}\label{sharp}

Let $M^\#$ be the usual sharp maximal function  of Fefferman
and Stein \cite{fs:fs},

$$M^\#f(x)=\sup_{Q\ni x}\inf_c \avgint_Q |f(y)-c|\, dy
\approx \sup_{Q\ni x}\avgint_Q |f(y)-f_Q|\, dy. $$

For a given  $\delta\in(0,\infty)$ we will use the notation 
$$
M_\delta f(x)= \Big(M( |f|^{\delta})(x)\Big)^{1/\delta}.
$$
and 
$$
M_\delta^\# f(x) = \Big(M^{\#}( |f|^{\delta})(x)\Big)^{1/\delta}=
\Big( \sup_{x\in Q}  \avgint_Q | |f|^{\delta}  - |f|^{\delta}_Q| dx\Big)^{1/\delta}
$$

The following form of the classical result of Fefferman
and Stein \cite{fs:fs} (see also \cite{j:j}) will be useful for our purposes.

Let $0<p,\delta<\infty$  and let $w$ be a weight in $A_{\infty}$. Then,
there exists $C>0$ (depending on the $A_{\infty}$ constant of $w$),
such that
\begin{equation}\label{FSppstrong}
\int_{\mathbb R^n} (M_\delta f(x))^p \, w(x)dx\leq C\,\int_{\mathbb R^n}
(M^{\#}_\delta f(x))^p \, w(x)dx,
\end{equation}
for all function $f$ for which the left  hand side is finite.

\subsection{Extrapolation}

We begin this section by recalling a modern version of the Rubio de Francia extrapolation theorem.

\begin{theorem} \label{rfrf} Let $\mathcal F$ be a family of pairs of functions $(f, g)$ such that, for some $1\leq p_0<\infty$, every $w\in A_{p_0}$, every $(f,g)\in\mathcal F$, and an increasing function $\varphi$,
$$
\|g\|_{L^{p_0}(w)} \le \varphi([w]_{A_{p_0}}) \|f\|_{L^{p_0}(w)}.
$$
Then, for every  $1<p<\infty$  there exists and increasing function $\phi_p$ depending only $\varphi$ such that  
for every $w\in A_{p}$ and every $(f,g)\in\mathcal F$, 
$$
\|g\|_{L^{p}(w)} \le \phi_p([w]_{A_{p}}) \|f\|_{L^{p}(w)}.
$$
\end{theorem}

 The dependence of $\phi_p$ on $\varphi$ and $p$ can be tracked explicitly; see \cite[Theorem 3.1]{d:d}.
In fact, we shall require the following more specific  version which can be found in  \cite{cst:cst}.

\begin{theorem}\label{RdeF-bis}
Let $W>0$ and let $\mathcal F$ be a family of pairs of measurable functions so that, for some $p_0>1$ 
$$
\varphi(W):=\sup_{\|w\|_{A_{p_0}\le W}}\sup_{(f,g)\in \mathcal F} \frac{\|g\|_{L^{p_0}(w)}}{  \|f\|_{L^{p_0}(w)}} <\infty, 
$$ 
then, for every $p>1$ there exist constants $C_1$ and $C_2$  so that if $M>0$ satisfies that $C_2 M^{\max(1, \frac{p_0-1}{p-1})}=W$, then 
$$
\sup_{\|w\|_{A_{p}\le M}}\sup_{(f,g)\in \mathcal F}  \frac{\|g\|_{L^{p}(w)}}{  \|f\|_{L^{p}(w)}} \le 
C_1\varphi(W)<\infty.
$$ 
\end{theorem}

 We will also use an extrapolation theorem based on the $A_{\infty}$ class in contrast to the usual $A_{p_0}$ condition, $p_0\in(1,\infty)$. 
This approach originated in \cite{cmp1, MPT}   and was subsequently expanded in \cite{CGMP}.

\begin{theorem} \label{ExtrapAinfty} Let $\mathcal F$ be a family of pairs of functions $(f, g)$ such that, for some $0< p_0<\infty$, every $w\in A_{\infty}$ and every $(f,g)\in\mathcal F$, 
$$
\|g\|_{L^{p_0}(w)} \leq \varphi([w]_{A_{\infty}}) \|f\|_{L^{p_0}(w)}
$$
 for some  increasing function $\varphi$.  Then for every $0<p<\infty$ there exists 
  an increasing  function $\phi_p$ depending only $\varphi$   such that for every $w\in A_{\infty}$ and every $(f,g)\in\mathcal F$, 
$$
\|g\|_{L^{p}(w)} \le \phi_p([w]_{A_{\infty}}) \|f\|_{L^{p}(w)}. 
$$
\end{theorem}

A theory of extrapolation in weighted mixed Lebesgue spaces has been   previously  developed by Kurtz.
In fact, the  following theorem was proved in \cite{k:k} but we include a simpler proof as a direct application of Theorem \ref{rfrf}.

\begin{theorem} \label{ext} Let $\mathcal F$ be a family of pair of functions $(f, g)$ defined on $\mathbb R^{n+m}$ such that, for some $p_0\ge 1$, every $w\in A_{p_0}(\mathbb R^{n+m})$ and every $(f,g)\in\mathcal F$, 
\begin{equation}\label{rf2}
\|g\|_{L^{p_0}_w(\mathbb R^{n+m})} \le \varphi([w]_{A_{p_0}}) \|f\|_{L^{p_0}_w(\mathbb R^{n+m})},
\end{equation}
for some increasing function $\varphi$. 
Then, for every $1<p,q<\infty$, there exist a function $\psi_{p,q}$ such that for every  $u\in A_p(\mathbb R^n)$ and  every $v\in A_q(\mathbb R^m)$
$$
\|g\|_{L^p_uL^q_v} \le \psi_{p,q}([u]_{A_p}, [v]_{A_q}) \|f\|_{L^p_uL^q_v}, $$
for all $ (f,g)\in\mathcal F$.
\end{theorem}

\begin{proof}    Fix   { $q>1$}. By Theorem \ref{rfrf}, we have that \eqref{rf2} implies that, for every $w\in A_q(\mathbb R^{n+m})$, 
{
$$
\|g\|_{L^{q}(w)} \le \phi_q([w]_{A_{q}}) \|f\|_{L^{q}(w)}.
$$
Let now  fix $v\in A_q(\mathbb R^m)$ and let us consider the function
$$
G(x)= \left(\int_{\mathbb R^m} |g(x,y)|^q v(y) dy\right)^{1/q}, 
$$
 and similarly define $F$ by replacing  $g$ with $f$. Then, for every   $u\in A_q(\mathbb R^n)$,  we clearly have that $u\otimes v\in A_q(\mathbb R^{n+m})$ with $[u\otimes v]_{A_{q}}\le [u]_{A_{q}}[ v]_{A_{q}}$, and hence 
\begin{align*}
\|G\|_{L^q_u(\mathbb R^m)} = \|g\|_{L^q_{u\otimes v}(\mathbb R^{n+m})} 
& \le \phi_q([u\otimes v]_{A_q}) \|f\|_{L^q_{u\otimes v} (\mathbb R^{n+m})}\\
& \le  \phi_q([u]_{A_{q}}[ v]_{A_{q}}) \|F\|_{L^q_u(\mathbb R^m)}\\
& \le  \Phi_{q,[v]_{A_{q}}}([u]_{A_{q}})\|F\|_{L^q_u(\mathbb R^m)},
\end{align*}
    where $\Phi_{q,[v]_{A_{q}}}([u]_{A_{q}})= \phi_q([u]_{A_{q}}[ v]_{A_{q}})$ is an increasing function of $[u]_{A_{q}}$.
Using Theorem \ref{rfrf}, we conclude that for every $1<p<\infty$ there exists an increasing function $\Psi_{p}$ depending on  $\Phi_{q,[v]_{A_{q}}}$ such that for every $u\in A_p(\mathbb R^n)$, 
$$
\|g\|_{L^p_u L^q_v(\mathbb R^{n+m})} =\|G\|_{L^p_u} \le \Psi_{p}([u]_{A_{p}}) \|F\|_{L^p_u}=\|f\|_{L^p_uL^q_v(\mathbb R^{n+m})}.
$$
}
{ Observe that $\Psi_{p}([u]_{A_{p}})$ is really a function  $\psi_{p,q}([u]_{A_p}, [v]_{A_q}) $} and the result follows. 
\end{proof}

{ We note that using the precise estimates in \cite[Theorem 3.1]{d:d} one can check that $\psi_{p,q}$ is actually increasing in each variable.}

The following result (also proved in \cite{k:k}) follow as an easy corollary of Theorem \ref{ext}.

\begin{corollary}  \label{max}   Let $M$ be the Hardy-Littlewood maximal operator in $\mathbb R^{n+m}$ then, for every  $1<p,q<\infty$, 
  $u\in A_p(\mathbb R^n)$ and  $v\in A_q(\mathbb R^m)$, 
$$
M:L^p_uL^q_v \longrightarrow L^p_uL^q_v,
$$
with norm controlled by $\varphi_{p,q}([u]_{A_p}, [v]_{A_q})$, with $\varphi_{p,q}$ an increasing function. 
\end{corollary}

 Another corollary of Theorem \ref{ext} is an extension of the Fefferman-Stein inequality \eqref{FSppstrong}
within the context of weighted mixed   Lebesgue spaces.

\begin{corollary}   Let $M$ and $M^{\#}$ be the Hardy-Littlewood maximal operator and Fefferman and Stein sharp maximal function in $\mathbb R^{n+m}$.  Let $0<p,q<\infty$, $\delta>0$ and let   $u\in A_{\infty}(\mathbb R^n)$ and  $v\in A_{\infty}(\mathbb R^m)$. Then there exists 
a function $\psi$ increasing in each variable such that 
$$
\|M_\delta f\|_{L^p_uL^q_v} \le \psi_{p,q}([u]_{A_{\infty}}, [ v]_{A_{\infty}})  \|M_\delta^\#f\|_{L^p_uL^q_v}.
$$
\end{corollary}

\begin{proof}  

By the classical result of Fefferman-Stein \eqref{FSppstrong} for any $0<p,\delta<\infty$, any $w\in A_{\infty}(\mathbb R^{n+m})$  
$$
\|M_\delta f\|_{L^{q}(w)} \leq c_{q,n}\, [w]_{A_{\infty}} \|M_\delta^\# f\|_{L^{q}(w)}.
$$
Now, let $v\in A_{\infty}(\mathbb R^m)$. As before, define the function
$$
G(x)= \left(\int_{\mathbb R^m} |M_{\delta}f(x,y)|^q v(y) dy\right)^{1/q}, 
$$
and similarly define $F$ by 
$$
  F(x)= \left(\int_{\mathbb R^m} |M^{\#}_{\delta}f(x,y)|^q v(y) dy\right)^{1/q}.
$$
Observing that for every   $u\in A_{\infty}(\mathbb R^n)$ and    $v\in A_{\infty}(\mathbb R^m)$ we clearly have that $u\otimes v\in A_{\infty}(\mathbb R^{n+m})$.  Moreover, if we denote by $Q_n$ cubes in $\mathbb R^n$ and similarly $Q_m$, then 
 
\begin{eqnarray*}
[u\otimes v]_{A_{\infty}}\! &\le& \! \sup_{Q_n\times Q_m} \frac 1{u(Q_n) v(Q_m)} \int_{Q_n\times Q_m} M\Big(\chi_{Q_n\times Q_m} u\otimes v\Big)(x,y)dx 
\\
&\le&\!
\sup_{Q_n\times Q_m} \frac 1{u(Q_n) v(Q_m)} \int_{Q_n\times Q_m} M (\chi_{Q_n}u)(x) M (\chi_{  Q_m}v  )(y)dx 
\\
&\le &[u]_{A_{\infty}}[ v]_{A_{\infty}},
\end{eqnarray*}
and hence, we have that 
\begin{align*}
\|G\|_{L^q_u(\mathbb R^m)} = \|M_{\delta}f\|_{L^q_{u\otimes v}(\mathbb R^{n+m})} 
& \leq c_{n,m,q}\, [u\otimes v]_{A_{\infty}} \|M^{\#}_{\delta}f\|_{L^q_{u\otimes v} (\mathbb R^{n+m})}\\
& \leq c_{n,m,q}\, [u]_{A_{\infty}}[ v]_{A_{\infty}} \|F\|_{L^q_u(\mathbb R^m)}. 
\end{align*}
Using Theorem \ref{ExtrapAinfty}, we conclude that, for every $v\in A_{\infty}(\mathbb R^m)$ and every $u\in A_{\infty}(\mathbb R^n)$, 
\begin{eqnarray*}
\|M_{\delta}f\|_{L^p_u L^q_v(\mathbb R^{n+m})} &=&
\|G\|_{L^p_u} \le \psi_{p,q}([u]_{A_{\infty}}, [ v]_{A_{\infty}}) \|F\|_{L^p_u}
\\
&=&\|M^{\#}_{\delta}f\|_{L^p_uL^q_v(\mathbb R^{n+m})},
\end{eqnarray*}
concluding the proof of the corollary. 
\end{proof}

\section {Compactness in weighted mixed Lebesgue spaces}

  We state in this section a version of Fr\'echet-Kolmogorov compactness theorem for weighted mixed Lebesgue spaces setting.  The following result, was proved by Clop and Cruz \cite{cc:cc}.

\begin{theorem} \label{ccc}
Let  $1<p< \infty$ and $u\in A_p$  Then $\mathcal F\subset L^p_u$ is relatively compact in $L^p_u$ if  the following conditions hold. 
\begin{enumerate} 

\item There exists $K>0$ such that, for every $f\in \mathcal F$, $\|f\|_{ L^p_u} \le K$.

\item For every $\varepsilon>0$,  there exists $R>0$ such that
$$
\|\chi_{B(0, R)^c}f \|_{L^p_u} <\varepsilon,  $$
for all $f\in \mathcal F$.

\item For every $\varepsilon>0$, there exists $\delta>0$ so that, for every $h$,  $|h|\le \delta$, 
$$
\|f- f(\cdot+h)\|_{ L^p_u}\le \varepsilon, 
$$
for all $f\in \mathcal F$.
\end{enumerate}
\end{theorem}

 There are actually several more general versions of this result which do not requiere the $A_p$ membership of the weights. A very general version in the context of BFS was obtained by Guo and Zhao \cite{GZ:GZ}. We refer the reader to that article for additional previous works. The direct extension to weighted mixed Lebesgue spaces below can be obtained adapting the arguments of  \cite{cc:cc}. The result, however, is a particular case of \cite[Theorem 3.1]{GZ:GZ} in the context of BFS and hence its proof is omitted.

\begin{theorem} \label{compact-mixed}Let $1<p,q < \infty$, $u\in A_p$, and $v\in A_q$.  Then $\mathcal F\subset L^p_uL^q_v$ is relatively compact in $L^p_uL^q_v$ if and only if the following conditions hold.
\begin{enumerate} 

\item There exists $K>0$ such that, for every $f\in \mathcal F$, $\|f\|_{ L^p_uL^q_v} \le K$.

\item For every $\varepsilon>0$,  there exists $R>0$ such that
$$
\|\chi_{B(0, R)^c}f \|_{L^p_uL^q_v} <\varepsilon,  $$
for all $f\in \mathcal F$.

\item For every $\varepsilon>0$, there exists $\delta>0$ so that, for every $h=(h_1, h_2)$,  $|h|\le \delta$, 
$$
\|f- f(\cdot+h_1, \cdot+h_2)\|_{ L^p_uL^q_v}\le \varepsilon, 
$$
for all $f\in \mathcal F$.
\end{enumerate}
\end{theorem}

We observe that the necessity of the conditions is not stated in \cite{GZ:GZ} but it follows in this context as in \cite[Section 10, Theorem 2]{bp:bp}.

We now extend to the mixed Lebesgue setting the notion of uniformly compactness  introduced   in \cite{cst:cst}.

\begin{definition}  \label{unif} An operator $T$ is uniformly compact in $ L^p_uL^q_v$ with respect to $(A_p, A_q)$ if 
there exist an increasing  function  $\varphi: (0, \infty)\to (0, \infty)$ such that  the following hold. 
\begin{enumerate} 
\item For every every $K\ge 1$, and  every $u\in A_p$ and  $v\in A_q$ with $\max( [u]_{A_p}\, [v]_{A_q})\le  K$ 
$$
\sup_{|f\|_{ L^p_uL^q_v}\le 1}\| Tf\|_{L^p_uL^q_v}\le  \varphi(K),
$$

\item  For every $\varepsilon>0$, every $K\ge 1$, there exists $R=R(\varepsilon, K)>0$ such that
$$
\sup_{\max( [u]_{A_p}\, [v]_{A_q})\le  K}\sup_{\|f\|_{ L^p_uL^q_v}\le 1}\|\chi_{B(0,R)^c} Tf\|_{L^p_uL^q_v}\le   \varphi(K)\varepsilon,
$$

\item For every $\varepsilon>0$ and every $K\ge 1$, there exists $\delta=\delta(\varepsilon,K)>0$ so that, for every  $h=(h_1,h_2)$ with $|h|\le\delta$, 
$$
\sup_{\max( [u]_{A_p}\, [v]_{A_q})\le  K} \sup_{\|f\|_{ L^p_uL^q_v}\le 1}\|Tf-Tf(\cdot+h_1, \cdot+h_2)\|_{ L^p_uL^q_v}\le  \varphi(K) \varepsilon.
$$
\end{enumerate}
\end{definition}

Clearly if an operator is uniformly compact it is also compact, since by Theorem \ref{compact-mixed} the image of the unit ball under $T$ is   precompact. 

  As in \cite{cst:cst}, we will also say that $T$ is uniformly compact in $L^p_{w}(\mathbb R^{n+m})$ if the conditions in Definition \ref{unif} hold with the norm in $L^p_w (\mathbb R^{n+m})$ and all $ w \in A_p(\mathbb R^{n+m})$.

\begin{remark} \label{3.4}
We also say that an operator $T$ is uniformly approximated with respect to $(A_p, A_q)$ by a sequence of operators $\{T_j\}_j$, if  there exist a  function  $\varphi: (0, \infty)\to (0, \infty)$ such that given 
$\varepsilon >0$ and $K\ge1$, there exists $j_0=j_0(\varepsilon, K)$ such that  for all $j>j_0$ 

$$
\sup_{\max([u]_{A_p}, [v]_{A_q})\le K}\sup_{|f\|_{ L^p_uL^q_v}\le 1}\| Tf - T_jf\|_{L^p_uL^q_v}\le  \varphi(K) \varepsilon,
$$
It is easy to see that if  a sequence of operators $\{T_j\}_j$ satisfies that each $T_j$ is uniformly compact in $ L^p_uL^q_v$ with respect to $(A_p, A_q)$   and  uniformly  approximates $T$ with respect to $(A_p, A_q)$, then 
$T$ is not only compact (as the limit of compact operators) but also uniformly compact in $ L^p_uL^q_v$ with respect to $(A_p, A_q)$.  In fact (1) and (2) in Definition \ref{unif} follow at once, while for (3)
we simply observe that, for any $g\in L^p_uL^q_v$,  
$$
\|g(\cdot+h_1, \cdot+h_2)\|_{ L^p_uL^q_v} = \|g\|_{L^p_{\tau_{h_1u}}L^q_{\tau_{h_2v}}},
$$
where $\tau_hu(x)= u(x-h)$ and that for any weight in $A_q$, $[\tau_hu]_{A_q}=[u]_{A_q}$. We leave the details to the reader.
\end{remark}

The following compact extrapolation result was proved in \cite{cst:cst}.

\begin{theorem} \label{ext-lp} If $T$ is an  operator  such that, for some $p_0>1$ 
$$
T:L^{p_0}(w) \longrightarrow L^{p_0}(w)
$$
is uniformly compact  with respect to $A_{p_0}$,  then for every $1<p<\infty$,  
$$
T: L^p_u \longrightarrow  L^p_u
$$
is uniformly  compact with respect to  $A_p$.
\end{theorem}

Our main theorem now reads as follows. 

\begin{theorem}\label{main}  If $T$ is an  operator  such that, for some $p_0>1$ 
$$
T:L^{p_0}(w) \longrightarrow L^{p_0}(w)
$$
is uniformly compact  with respect to $A_{p_0}(\mathbb R^{n+m})$,  then for every $p, q>1$,  
$$
T: L^p_uL^q_v \longrightarrow  L^p_uL^q_v
$$
is uniformly  compact with respect to  $(A_p, A_q)$.
\end{theorem}

\begin{proof}  
By the compact extrapolation result in \cite{cst:cst} we immediately get that for all $1<p<\infty$, 
$$T:L^{p}(w) \longrightarrow L^{p}(w)
$$
is uniformly compact  with respect to $A_{p}(\mathbb R^{n+m})$.

Let us now fix $K\ge 1$, $p, q>1$ and $u\in A_{p}$ and $v\in A_q$ such that $\max( [u]_{A_p}, [v]_{A_q})\le K$.  We need to verify the  three conditions in Definition \ref{unif}. We will see that they follow  by  applying three times Theorem \ref{ext}.

For 
({\it 1)},   let $\|f\|_{L^p(L^q w)}\le 1$. Then, 
$$
\|Tf\|_{L^p_uL^q_v} \le \phi([u]_{A_p}, [v[_{A_q}) \le \varphi(K), 
$$
if we use  $(f,g)= (f,T(f))$ in Theorem \ref{ext}.

For {\it (2)},  since for every  $K>0$ and  every $\varepsilon>0$ there exists $R>0$ such that
$$
\bigg(\int_{|x|\ge R} |Tf(x)|^{p_0} w(x) dx\bigg)^{1/p_0} \lesssim  \varphi(K) \varepsilon \bigg(\int_{\mathbb R^n} |f(x)|^{p_0} w(x) dx\bigg)^{1/p_0}, 
$$
whenever $[w]_{A_{p_0}}\le K$,  then  by Theorem \ref{ext} with 
$$(f,g) = (f,\chi_{B(0,R)^c} Tf),$$
 for every $u$ and $v$ so that $\max( [u]_{A_p}, [v]_{A_q})\le K$, there exists $R>0$ so that 
$$
\|\chi_{B(0,R)^c} Tf\|_{L^p_uL^q_v}\lesssim   \psi(K, K)  \varepsilon  \|f\|_{L^p_uL^q_v},
$$
and the result follows.

Similarly, for {\it(3)}  we use the extrapolation theorem with 
$$(f,g) = (f, (Tf (x + t) - Tf (x))). $$
\end{proof}

\section{Applications}

\subsection{Commutators of singular integrals and point-wise multiplication}

There exist a very extensive and ever growing literature about the boundedness and compactness of commutators of various singular integrals and point-wise multiplication by function in a variety of contexts, including linear, multilinear, and weighted settings. In general commutation with function in $BMO$ produces bounded operators, while commutation with functions in $CMO$ produces compact ones. We refer the, reader for example, to \cite{bmmst:bmmst} and \cite{bt:bt} for comprehensive approaches and references. 
  Here $BMO$ is the usual space of functions of 
bounded mean oscillations, while $CMO$ is defined as the closure in the topology of $BMO$ of the space $C_c^{\infty}$ of smooth functions with compact support.

The first to prove the compactness of the commutators of Calder\'on-Zygmund operators with $CMO$-functions in a weighted $L^p$ setting were Clopp and Cruz in the already cited work \cite{cc:cc}.  In order to apply our theory to extrapolate this result to weighted mixed Lebesgue spaces we need to verify that those authors proved that the operators in question are actually uniformly compact in $L^p$ with respect to $A_p$.  This can be done by inspection of their proof. We will not repeat all of their computations here but, for the benefit of the reader, we shall summarize some key points where estimates depending on the classes of weights are obtained.

For simplicity and the purposes of this article, we define a Calder\'on-Zygmund operator $T$ to be a  bounded  linear operator on $L^2(\rn)$   that  can be written as 
$$
Tf(x)=\int_{\rn}  \KK(x,y)f(y) \ dy
$$
for $x\notin {\rm supp} \, f$,   where  the kernel $\KK$ satisfies 
\begin{equation}\label{regularity}
|\partial^{\beta}\KK(x,y)|\leq C_\KK|x-y|^{-n-|\beta|}, \quad |\beta|\leq 1.
\end{equation}
We will often write in such a case $T\in CZO$.

 The compactness in the following result can be obtained using the BFS machinery as done in \cite[Theorem 6.1]{LN:LN} once the boundedness  of $M$ in mixed Lebesgue spaces is stablished (for example in Corollary~\ref{max} or \cite{k:k}).  We want to emphasize again that with the theory just developed we obtain uniform compactness with straightforward computations.

\begin{theorem} \label{CommCompa}
Let $T \in CZO$ and $b\in CMO$ and consider the commutator
$$
C_bf=[T,b]f = T(bf)- bTf.
$$
Then, for every $u\in A_p(\rn)$ and every $v\in A_q(\mathbb R^m)$
$$
C_b: L^p_uL^q_v \longrightarrow L^p_uL^q_v
$$
is   uniformily   compact. 
\end{theorem}

\begin{proof}  Let $d:=n+m$. To prove that $C_b: L^p_uL^q_v \longrightarrow L^p_uL^q_v$ is compact whenever 
$u\in A_p(\rn)$ and $v\in A_q(\mathbb R^m)$, we will use  the extrapolation Theorem \ref{main} with exponent $p_0=2$, namely  we have to prove that 
$$
C_b: L^{2}(w) \longrightarrow L^{2}(w)
$$
is uniformly compact  with respect to $w\in A_{2}(\mathbb R^d)$.

We shall follow the proof of \cite{cc:cc} checking all the details that remains open in our setting.

\noindent
{\bf Step 1:} 
Let us approximate $b \in CMO$ by functions $b_j \in C^\infty_c$.  By a result in \cite{p:p}, we have that,  for  $w\in A_2(\mathbb R^d)$, 
$$
\|C_bf\|_{L^2(w)} \lesssim \|b\|_* \|M^2f\|_{L^2(w)},
$$
where $|b\|_*$ is the $BMO$ norm. Therefore, 
$$
\|C_bf - C_{b_j}\|_{L^2(w)} \lesssim \|b-b_j\|_* [w]_{A_2}^{2}\|f\|_{L^2(w)},
$$
and $C_{b_j} \to C_b$ uniformly on  $[w]_{A_{2}(\mathbb R^d)}\le K$, with $K$ a fixed and arbitrary constant. So it is enough to verify that the $C_{b_j}$ are uniformly compact.


\noindent
{\bf Step 2:} Next, for each $j$ the operators $C_{b_j}$ are then approximated by $C^\eta_{b_j}$ which is the commutator obtained by replacing the kernel $\KK$ of $T$ by a smooth truncated version $\KK^\eta$, where $\eta$ is an appropriate parameter. The authors of \cite{cc:cc} show that 
$$|Cb_j f (x) -Cb_j^\eta (x)| \lesssim  \eta \| \nabla b_j\|_\infty M f (x),$$ 
from which it follows that $Cb^\eta_j \to Cb_j$ in the uniform sense as $\eta \to 0$. 


\noindent
{\bf Step 3:}  For a fixed $b$ and $\eta$,  it is then needed to prove the  compactness of $C_b^\eta$. To this end, we will  use the version of the Fr\'echet-Kolmogorov result in Theorem \ref{ccc}.  Condition {\it (1)}  is uniformly satisfied by the already mentioned estimates on the commutator.   To verify {\it (2)} and 
 {\it (3)}, Clop and Cruz separate the quantities in questions into several terms. They obtained several pointwise estimates ultimately controlled by $M$ and $T_*$ (the maximal truncated singular integral), see \cite[p.98]{cc:cc}. These terms clearly produce uniform estimates with respect to $A_{2}(\mathbb R^d)$.

 Finally there are some terms which are dominated by  expressions that can be  controlled (up to constants) by
   
 \begin{align} \label{E1}
  \int_{3B_0}   &  \left( \int_{B_x\cap 4B_0} w(y)^{- 1} \, dy   \right)  w(x) \, dx   \nonumber \\
  &  + \,\,  \int_{^c(3B_0)}   \left( \int_{B_x\cap 4B_0} w(y)^{- 1} \, dy\right)   \frac {w(x)} {|x|^{2d}} \, dx, 
 \end{align} 
where $B_0$ is a fixed ball centered at the origin of radius $R_0$, which we can take bigger than 1, and  $B_x$ is a set depending on $x$ but with size $|B_x| \approx h$ where $|h| \to 0$;
and
\begin{equation} \label{E2}
\left( \int_{|x|<R_0} w(x)^{-1} dx\right)^{1/2} \left( \int_{|x|>R} \frac {w(x)} {|x|^{2d}} dx \right)^{1/2},
\end{equation} 
with  $R>R_0$ fixed, $R\to \infty$;  see \cite[pp.99-100]{cc:cc}.  We need to show that these terms, which obviously tend to zero when $h\to 0$ in \eqref{E1} and  $R\to \infty$ in \eqref{E2}, do so uniformly on $[w]_{A_2}$, whenever $[w]_{A_2}\le K$ with $K$ any fixed number.


\noindent
{\bf Estimates for \eqref{E1}}: Note that by H\"older's inequality we get that, for any $r>1$, 
$$
\left( \int_{B_x\cap 4B_0} w(y)^{-1} \, dy\right)  \le \left( \int_{ 4B_0} w(y)^{-r} dy\right)^{1/r}  |B_x|^{1/r'}. 
$$
To control \eqref{E1}, we shall choose $r_1>1$ and $r_2>1$ appropriately and  estimate 
\begin{equation} \label{E3}
I := \left( \int_{4B_0} w(x)^{-r_1} dx\right)^{1/r_1} \left( \int_{ 3B_0} w(x) dx\right)\, \,  |h|^{1/r_1'}
\end{equation} 
and
\begin{equation} \label{E32}
II := \left( \int_{4B_0} w(x)^{-r_2} dx\right)^{1/r_2} 
 \int_{ (3B_0)^c} \frac {w(x)} {|x|^{2d}} dx \,\,  |h|^{1/r_2'}.
\end{equation} 
To estimate $I$ in \eqref{E3}, and using that $w\in A_2$ implies $w^{-1}\in A_2$ with same $A_2$ constant,  we can choose  $q= 2-\varepsilon$ with 
$$\varepsilon= \frac{1}{c_d [w]_{A_{2}} }= \frac{1}{c_d [w^{-1}]_{A_{2}} }  \leq \frac{1}{c_d [w^{-1}]_{A_{\infty}} } \text{\, and \, }
%
[w]_{A_q} \leq 2 [w]_{A_2}
$$
by the precise openness property of the $A_2$ class of weights (see \eqref{rprime} and \eqref{eq:Ap-e}).
It  follows that with $r_1=\frac{q'}{q}>1$,
$$  
I \le \left(  \avgint_{4B_0} w(x)^{-r_1} dx\right)^{1/r_1}   \avgint_{ 4B_0} w(x) dx\, \,  |4B_0|^{1+1/r_1} |h|^{1/r_1'} ,
$$
$$
\leq [w]_{A_q}  |4B_0|^{1+1/r_1} |h|^{1/r'_1}  \leq 2    [w]_{A_2}  |4B_0|^{1+1/r}  |h|^{1/r'_1}. 
$$
Now, fixed $K$  and $[w|_{A_2}\le K$, it can be check that 
$$
r_1=\frac1{q-1}=\frac 1{1-\varepsilon} \implies  r_1'=\frac 1{\varepsilon}=c_d[w]_{A_2} \le c_d K,
$$
and hence
$$
I \le K |4B_0|^2  |h|^{1/(c_d K)},
$$
which  tends to zero  as $|h| \to 0$, uniformly on  $[w]_{A_2}\le K$.

We now estimate $II$ in  \eqref{E32}. Setting $R_1$ equal to the radius of $3B_0$, 
$$
\int_{|x|>R_1} \frac {w(x)} {|x|^{2d}} dx = \sum_{j=0}^\infty \int_{2^jR_1<|x|\leq 2^{j+1}R_1} \frac {w(x)} {|x|^{2d}} dx
$$
$$
\leq  \sum_{j=0}^\infty  (2^jR_1)^{-2d} w(B(0,2^{j+1}R_1).
$$
 Let $q$ be the exponent from the precise openness property as above. Then, by 
standard properties of the $A_p$ class of weights (see \eqref{eq:prop-Ap-set})  we have 
that for any pair of balls $B_1\subset B_2$, that
\begin{equation}\label{Approp}
w(B_2)
\leq [w]_{A_q}
\left (\frac{|B_2|}{|B_1|} \right )^q w(B_1)
\end{equation}
and we can continue with 
$$
 \lesssim [w]_{A_q}\sum_{j=0}^\infty  (2^jR_1)^{-2d} (2^{j+1}R_1)^{qd} w(B(0,1))
 \lesssim [w]_{A_2} w(B(0,1)) R_1^{d(q-2)},
$$
since $q<2$. 

Now, taking $r_2 = 1+\frac{1}{2^{d+1}[w]_{A_2}-1} \leq r_{w^{-1}}$ from the 
precise reverse H\"older property \eqref{romega}, we have as in \eqref{eq:RHI}
$$
\left(  \avgint_{4B_0} w(x)^{-r_2} dx\right)^{1/r_2} \leq 2
\avgint_{4B_0} w(x)^{-1} dx,
$$
and we can use again \eqref{Approp} and $[w^{-1}]_{A_2}=[w]_{A_2}$  to obtain
$$
\avgint_{4B_0} w(x)^{-1} dx \lesssim [w]_{A_2}   w^{-1}(B(0,1)) |4B_0|.
$$
So all together we obtain that, for some constant $C_K$,

\begin{eqnarray*}
II &\lesssim& [w]_{A_2}^{2} |4B_0|^{1+\frac1r_2}  R_1^{n(q-2)}     w(B(0,1)) w^{-1}(B(0,1))|h|^{1/r_2'} 
 \\
&\lesssim& [w]_{A_2}^{3} |4B_0|^{2}  R_1^{d(q-2)}   |h|^{C_K},
\end{eqnarray*}
which again tends to zero uniformly in $[w]_{A_2}\le K$ as $|h| \to 0$.


\noindent
{\bf Estimate for \eqref{E2}}: %

We have, as before,  that since $q<2$, 
$$
\int_{|x|>R} \frac {w(x)} {|x|^{2d}} dx = \sum_{j=0}^\infty \int_{2^jR<|x|\leq 2^{j+1}R} \frac {w(x)} {|x|^{2d}} dx
 \lesssim  [w]_{A_2} w(B(0,1)) R^{d(q-2)}.
$$
Also, for the first factor in \eqref{E2} we use again \eqref{Approp}, 
$$
\int_{|x|<R_0} w(x)^{-1} dx \lesssim [w]_{A_2} R_0^{qd} w^{-1}(B(0,1)).
$$
So combining all the estimates,  \eqref{E2} is controlled up to a constant by 
$$
[w]_{A_2}^3 R^{d(q-2)} R_0^{qd}
$$
which tends to zero  uniformly on $[w]_{A_2}\le K$ as $R\to \infty$.
\end{proof}
\subsection{Pseudodifferential operators}

The first example of non-trivial compact pseudodifferential  operators on weighted Lebesgue spaces that we are aware of was obtained in 
\cite{cst:cst}.  The authors combined conditions on the symbol considered by Cordes \cite{c:c} to produce compact operators on $L^2$ with additional conditions that provide a priori boundedness in $L^p$ for $p\neq 2$.  More precisely the following was obtained in \cite{cst:cst}.

\begin{theorem}\label{pseudos}
Let $$T_\sigma f(x) =  \int_ {\rd} \sigma(x,\xi) \widehat f(\xi)  e^{ix \xi }\, d\xi$$ 
be a pseudodifferential operator  with symbol $\sigma$ satisfying the conditions 
\begin{equation}\label{conditions}
|\partial_x^\alpha \partial_{\xi}^\beta \sigma(x,\xi)| \lesssim C_{\alpha,\beta}(x,\xi) (1+|\xi|)^{-|\beta|},
\end{equation}
for all  $|\alpha|, |\beta|$, where $C_{\alpha,\beta}$ is a bounded function which tends to zero as $|x|^2+|\xi|^2 \to \infty$.
 Then $T_\sigma$ is compact in $L^p(w)$ for all $1<p<\infty$ and all $w \in A_p$.
\end{theorem}

Although the authors stated the conclusion as $T_\sigma$ being compact, it is actually uniformly compact with respect to $A_2$. In fact, the proof of the theorem involves the construction of a sequence of compact Calder\'on-Zygmund operators $\{T_j\}_j$ converging to $T_\sigma$ in the ``Calder\'on-Zygmund norm" defined by
$$
\|T\|_{CZO} = \|T\|_{L^2} + C_K,
$$
where $C_K$ is the best constant in \eqref{regularity}. Moreover the $T_j$ were proved to be uniformly compact with respect to $A_2$ and hence by compact extrapolation uniformly compact with respect to all $A_p$, $1<p<\infty$.   However, it was also shown using the Fefferman-Stein inequality, see \cite[Remark 2.2 ]{cst:cst},  that 
$$ \int_{\rd} |(T_\sigma-T_j)({ f})(x)|^p w(x)\, dx 
  \lesssim \, \|T_\sigma-T_j\|_{CZO}\, \int_{\rd}     {M}({ f})(x)^{p} w(x)\, dx.
$$
It follows then than $T_j \to T_\sigma$ uniformly with respect to $A_p$. By Remark \ref{3.4} (analogously applied to $L^p$ instead of $L^pL^q$) it follows that $T_\sigma$ is uniformly compact in $L^p(w)$ with respect to $A_p$ for all $1<p<\infty$.
As a corollary of  Theorem \ref{main} we then obtain from Theorem \ref{pseudos} the following result.
\begin{theorem}
$T_\sigma$ is uniformly in compact on $L^p_uL^q_v$ for all $1<p,q<\infty$, $u\in A_p$, and $v\in A_q$.
\end{theorem}

{\bf Acknowledgement.} We would like to thank E. Lorist and Z. Nieraeth for valuable comments pointing out references \cite{GZ:GZ} and \cite{LN:LN}, which we overlooked on a first version of this work.

\end{document}